\documentclass[12pt]{amsart}

\usepackage{hyperref}
\hypersetup{nesting=true,debug=true,naturalnames=true}
\usepackage{graphicx,amssymb,upref,units}

\usepackage{tikz}

\let\<\langle
\let\>\rangle

\let\uml\"

\usepackage{amsmath,amssymb,fullpage,xy,mathrsfs,amsthm,amsxtra,graphicx,verbatim}
\usepackage{hyperref}
\hypersetup{colorlinks=true,urlcolor=magenta,citecolor=magenta,linkcolor=blue}
\usepackage{colonequals}
\usepackage{multirow}

\usepackage[mathcal]{euscript}

\newcommand{\GL}{\operatorname{GL}}

\newcommand{\C}{\mathbf{C}}

\newcommand{\ddef}{\colonequals}

\newcommand{\sgn}{\mathsf{sgn}}

\newcommand{\SSS}{\mathcal{S}}
\newcommand{\CCC}{\mathcal{C}}

\numberwithin{equation}{section}

\theoremstyle{plain}
\newtheorem{thm}[equation]{Theorem}

\newtheorem{lem}[equation]{Lemma}
\newtheorem{defn}[equation]{Definition}
\newtheorem{cor}[equation]{Corollary}
\newtheorem{prop}[equation]{Proposition}

\theoremstyle{remark}
\newtheorem{rmk}[equation]{Remark}

\newtheorem{exm}[equation]{Example}

 \input xy
 \xyoption{all}

\begin{document}

\title{Unisingular Specht Modules}

\author{John Cullinan}
\address{Department of Mathematics, Bard College, Annandale-On-Hudson, NY 12504, USA}
\email{cullinan@bard.edu}
\urladdr{\url{http://faculty.bard.edu/cullinan/}}

\begin{abstract}
Let $G$ be a finite group and $\rho:G \to \GL(V)$ a finite dimensional representation of $G$.  We say that $\rho$ is unisingular if $\det(1-\rho(g)) = 0$ for all $g \in G$.  Building on previous work in \cite{cullinan}, we consider the symmetric groups $S_n$ and prove that certain families of Specht modules are always unisingular as well as raise new questions for future study.
\end{abstract}

\maketitle

\section{Introduction}

Let $G$ be a finite group and $\rho:G \to \GL(V)$ a finite dimensional representation of $G$.  We call $\rho$ \textsf{unisingular} if $\det(1-\rho(g)) = 0$ for all $g \in G$. Unisingular representations have a multitude of applications in Lie theory and number theory. In arithmetic geometry, for instance, they are used for point-counting on algebraic curves and abelian varieties (see \cite{cz} for an example and overview).  It is a difficult problem in general to characterize, for a given family of groups, all of its unisingular representations; see, for example, \cite{z90} and  \cite{zalesski} for certain finite Chevalley and symplectic groups. Of course, the trivial representation is unisingular for any group,  so we naturally look to classify irreducible unisingular representations.  

In this paper we consider the symmetric groups $S_n$ 
and their ordinary representation theory.  The irreducible representations of $S_n$ are the \textsf{Specht modules} $\SSS^\lambda$, indexed by the partitions $\lambda$ of $n$ (we write $\lambda \vdash n$).  The $\SSS^\lambda$ fully describe the irreducible representations of the $A_n$ as well, though some modifications are required; we review these in Section \ref{background} below.  Since the number of conjugacy classes of a group is equal to the number of its ordinary absolutely irreducible representations, the conjugacy classes $\CCC_\mu$ of $S_n$ are indexed by partions $\mu \vdash n$ as well, corresponding to the cycle type of the $\pi \in \CCC_\mu \subseteq S_n$.  

Recently, in \cite{cullinan},  we focused on applications of unisingular representations to Jacobians of hyperelliptic curves.  In the process, we showed that the Specht modules $\SSS^{(n-2,2)}$ and $\SSS^{(n-2,1^2)}$ were unisingular for $n \geq 4$ and raised the conjecture that $\SSS^{(2^2,1^{n-4})}$ is unisingular if and only if $n \geq 5$ is even.  In this paper we generalize those results by showing that they fit into a larger family of unisingular Specht modules. Before stating our theorems we give a few examples that highlight some of the intricacies in determining which Specht modules are unisingular; we will provide proofs of these statements in the main body of the paper. 

\begin{exm}
Let $n$ be a positive integer.
\begin{enumerate}
\item The trivial representation $\mathbf{1} = \SSS^{(n)}$ is unisingular for $S_n$ for all $n \geq 1$.  Similarly, the sign representation $\sgn = \SSS^{(1^n)}$ is not unisingular once $n > 1$.
\item  The standard representation $\SSS^{(n-1,1)}$ is not unisingular for $S_n$ for any $n>1$; the characteristic polynomial on the $n$-cycles is $\sum_{j=0}^{n-1} x^j$, which does not vanish at 1.  However, every other permutation has 1 as an eigenvalue on $\SSS^{(n-1,1)}$.  Therefore, if $n$ is even (so $A_n$ does not contain any $n$-cycles) then $\SSS^{(n-1,1)}\big|_{A_n}$ is unisingular. 
\item The 5-dimensional representation $\SSS^{(2,2,2)}$ is not unisingular for $S_6$. The characteristic polynomial on the class $\mathcal{C}_{(6)}$ factors as the product of cyclotomic polynomials $\Phi_2(x)\Phi_3(x)\Phi_6(x)$.  This is the first example of a representation other than $\SSS^{(n-1,1)}$ or $\SSS^{(1^n)}$ that is not unisingular for $S_n$.
\item The 14-dimensional representations $\SSS^{(4,4)}$ and $\SSS^{(2,2,2,2)}$ are not unisingular for $S_8$ or $A_8$; note that $\SSS^{(2,2,2,2)} = \SSS^{(4,4)} \otimes \sgn$, hence $\SSS^{(4,4)} \big|_{A_n} = \SSS^{(2,2,2,2)}\big|_{A_n}$.   In both cases we observe that the characteristic polynomial on the class $\CCC_{(5,3)}$ factors as the product of cyclotomic polynomials $\Phi_3(x)\Phi_5(x)\Phi_{15}(x)$.  This is the first example of a representation other than $\SSS^{(n-1,1)}$ that is not unisingular for $A_n$. 
\end{enumerate}
\end{exm}

There are many factors to consider when trying to determine in full generality which Specht modules are unisingular: the parity of $n$, the length and type of partition $\lambda$, the factorization of characteristic polynomials into products of cyclotomic polynomials, etc.  We therefore take a modest approach to understanding unisingular Specht modules and restrict to certain special cases for the purposes of this paper.

We focus almost exclusively on the family of \textsf{Gamma-shaped} Specht modules.  These are the representations 
\[
\SSS^{(n-k,1^k)} = \wedge^k \SSS^{(n-1,1)},
\]
for $k=1,\dots,n-2$, and are so-named because the associated Young diagram of the partition $(n-k,1^k)$ of $n$ consists of a single hook; note that $k=0$ corresponds to the trivial representation $\mathbf{1}$ and the case $k = n-1$ corresponds to the sign representation $\sgn$.  Our main theorem is as follows.

\begin{thm} \label{mainthm}
Let $n\geq 5$ be a positive integer.  Then
\begin{enumerate}
\item $\SSS^{(n-1,1)}$ is not unisingular, and
\item $\SSS^{(n-k,1^k)}$ is unisingular for all $k=2,\dots,n-3$, and
\item $\SSS^{(2,1^{n-2})}$ is unisingular if and only if $n$ is even. 
\end{enumerate}
\end{thm}

\begin{rmk}
For small values of $n$, we can  determine which Specht modules are unisingular by computer calculation, regardless of whether they are Gamma-shaped or not.  We present these data in the final section of the paper. 
\end{rmk}

As an auxiliary result we prove the conjecture raised in \cite{cullinan}.

\begin{thm} \label{(2^2,1^{n-4})}
Let $n \geq 5$.  Then the Specht module $\SSS^{(2^2,1^{n-4})}$ is unisingular if and only if $n$ is even.
\end{thm}

We conclude this paper with some general observations for future study.  It is our hope that the methods of this paper will be useful in determining all unisingular Specht modules, regardless of $\lambda$.

\section{Notation and Background} \label{background}

For a detailed overview of the symmetric group and its representations, see \cite{fh}, \cite{james}, or \cite{sagan}; we adopt most of the notation of \cite{sagan} here.  Due to these existing thorough treatments, we will give a very brief overview of the symmetric group and its representations here, and only introduce the necessary definitions for the purposes of this paper.

\subsection{Partitions} Let $n$ be a positive integer. A \textsf{partition} $\lambda = (\lambda_1,\dots,\lambda_r)$ of $n$ is a vector of non-increasing positive numbers $\lambda_i$ such that $\sum \lambda_i = n$; if $\lambda$ is a partition of $n$, then we write $\lambda \vdash n$.  A convenient mechanism for studying partitions is the \textsf{Young tableau} $t_{\lambda}$ of shape $\lambda$ \cite[p.~20]{sagan}, a left-justified array with $\lambda_i$ boxes in the $i$th row.  If $\lambda \vdash n$ then we define $\lambda'$ to be the \textsf{conjugate} partition, whose associated $t_{\lambda'}$ is the reflection of $t_{\lambda}$ across the diagonal.  For example, if $\lambda = (6,6,4,1) \vdash 17$, then $\lambda' = (4, 3, 3, 3, 2, 2)$. 

The conjugacy classes (and thus the irreducible representations) of $S_n$ are indexed by the partitions of $n$.  We write $\mathcal{C}_\lambda$ for the conjugacy class of permutations $\pi$ that decompose as $\pi = \sigma_1 \cdots \sigma_r$, where the $\sigma_i$ are disjoint cycles of length $\lambda_i$. We write $\SSS^\lambda$ for the irreducible representation indexed by $\lambda$.

\subsection{Specht Modules}  let $\lambda \vdash n$.  The irreducible representations of $S_n$ are known as the Specht modules and the dimension of $\SSS^\lambda$ is denoted by $f^\lambda$, which is given by the well-known \textsf{hook length formula} \cite[Thm.~3.10.2]{sagan}.  Related to the Specht modules are the \textsf{permutation modules} $\mathcal{M}^\lambda$, which contain certain Specht modules as direct summands: $\mathcal{M}^\lambda = \oplus_\mu K_{\mu\lambda}\SSS^{\mu}$ (for more information see \cite[Thm.~2.11.2]{sagan}).  There are several special cases of this decomposition that are of particular interest to us.

\begin{itemize}
\item If $\lambda = (n)$, then $\mathcal{M}^{(n)} = \SSS^{(n)} = \mathbf{1}$, the trivial representation.
\item If $\lambda = (n-1,1)$, then $\mathcal{M}^{(n-1,1)}$ is the natural $n$-dimensional representation of $S_n$, $\SSS^{(n-1,1)}$ is the standard representation,  and 
\begin{align} \label{standard_decomp}
\mathcal{M}^{(n-1,1)} = \mathbf{1} \oplus \SSS^{(n-1,1)}.
\end{align}   
\item Similar to the previous case, one can show (see, for example, \cite[p.~14]{cullinan}) that
\begin{align} 
\mathcal{M}^{(n-2,2)} &= \SSS^{(n-2,2)} \oplus \mathcal{M}^{(n-1,1)} \label{n-2specht} \\
&= \SSS^{(n-2,2)} \oplus \mathcal{S}^{(n-1,1)} \oplus \SSS^{(n)}.
\end{align}
\item The 1-dimensional Specht module $\SSS^{(1^n)}$ is the sign representation $\sgn$.  It is well known that 
\begin{align} \label{sgn_decomp}
\sgn \otimes \SSS^{\lambda} = \SSS^{\lambda'}
\end{align}
for all $\lambda \vdash n$.  The representations $\SSS^\lambda$ and $\SSS^{\lambda'}$ are isomorphic upon restriction to $A_n$ since $\sgn$ is trivial on even permutations.
\item Continuing, the Specht module $\SSS^\lambda$ remains irreducible upon restriction to $A_n$ unless $t_\lambda  = t_{\lambda'}$.  In that case, $\SSS^\lambda\big|_{A_n}$ is the direct sum of two irreducible representations of the same dimension.  This completely characterizes the irreducible representations of $A_n$. 
\item If $\lambda = (n-k,1^k)$ for $k=1,\dots,n-1$, then $t_\lambda$ consists of a single hook; in these cases we say that $\lambda$ (or $\SSS^\lambda$) is {Gamma-shaped}. It is known that (see \cite[Ex.~4.6]{fh}) Gamma-shaped Specht modules are simply the exterior powers of $\SSS^{(n-1,1)}$:
\begin{align} \label{wedge}
\wedge^k \SSS^{(n-1,1)} = \SSS^{(n-k,1^k)}.
\end{align}
In the extreme case this recovers the well-known identity 
\[
\wedge^{n} \mathbf{1} = \wedge^{n}\SSS^{(n)} = \SSS^{(1^n)} = \sgn.
\]
It is the Gamma-shaped partitions that we are primarily concerned with in this paper. 
\end{itemize}

\subsection{Linear Algebra} To determine whether a given Specht module is unisingular, we need to show that the characteristic polynomials of all permutations vanish at $x=1$; we set some notation in anticipation of these calculations. Recall that characteristic polynomials are constant on conjugacy classes.

We will use the following facts from linear algebra frequently over the course of the paper.  If $A$ is a linear transformation on an $a$-dimensional vector space $V$ with eigenvalues $\alpha_j$, and $B$ is a linear transformation on a $b$-dimensional vector space $W$ with eigenvalues $\beta_k$, then the characteristic polynomial of $A \otimes B$ on $V \otimes W$ is 
\begin{align} \label{tensor_evalue}
c_{A \otimes B} (x) \prod_{j,k} (x-\alpha_j\beta_k).
\end{align}
Additionally, if $\wedge^pA$ denotes the induced linear transformation on the $\binom{a}{p}$-dimensional vector space $\wedge^pV$, then its characteristic polynomial is 
\begin{align} \label{ext_evalue}
c_{\wedge^pA}(x) = \prod_{i_1 < i_2 < \cdots < i_p} (x-\alpha_{i_1}\alpha_{i_2} \cdots \alpha_{i_p}).
\end{align}
Now we apply this formalism to the modules that concern us, with particular attention to the permutation module $\mathcal{M}^{(n-1,1)}$.

Fix a partition $\mu = (\mu_1,\dots,\mu_r) \vdash n$ and consider $\pi \in \mathcal{C}_\mu \subseteq S_n$; write 
\[
\pi = \sigma_1 \dots \sigma_r,
\]
where $\sigma_i$ is a cycle of length $\mu_i$.  Consider the $n$-dimensional permutation module $\mathcal{M}^{(n-1,1)}$.  Conjugating $\pi$ if necessary, we can assume that the action of $\pi$ on $\mathcal{M}^{(n-1,1)}$ is via the block matrix
\[
\begin{pmatrix} \fbox{$\sigma_1$} \\ & \fbox{$\sigma_2$} \\ && \ddots \\ &&& \fbox{$\sigma_r$}  
\end{pmatrix}, 
\]
which evidently has characteristic polynomial
\begin{align} \label{char_poly_std}
m_\mu(x) = \prod_{j=1}^r (x^{\mu_j} - 1).
\end{align}

From this, we obtain some simple but important results that we will use extensively throughout the paper.

\begin{prop} \label{std_evalue}
Suppose $\mu = (n_1,\dots,n_r) \vdash n$ and $\pi \in \mathcal{C}_\mu$.  Let $z_{\mu_i}$ be a primitive $\mu_{i}$-th root of unity.  Then the characteristic polynomial of $\pi$ on $\SSS^{(n-1,1)}$ has roots
\begin{align} \label{Emu}
\Lambda_\mu \ddef \lbrace \underbrace{1,\dots,1}_{(r-1)\text{ times}} \rbrace \cup \bigcup_{i=1}^r \lbrace z_{\mu_i},z_{\mu_i}^2,\dots, z_{\mu_i}^{\mu_i-1} \rbrace.
\end{align}
\end{prop}

\begin{proof}
By (\ref{char_poly_std}), the eigenvalues of $\pi$ on $\mathcal{M}^{(n-1,1)}$ are the $\mu_i$-th roots of unity, for each $i=1,\dots,r$.  Now apply (\ref{standard_decomp}). 
\end{proof}

\begin{cor} \label{neg_cor}
Let $n >1$ and let $\mu = (\mu_1,\dots,\mu_r) \vdash n$.  Let $\pi \in \mathcal{C}_\mu \subseteq S_n$.  If $\pi$ is an odd permutation then $\pi$ has eigenvalue $-1$ on $\mathcal{S}^{(n-1,1)} \otimes \sgn$ of multiplicity $r-1$.
\end{cor}

\begin{proof}
By Proposition \ref{std_evalue}, $\pi$ has eigenvalue 1 of multiplicity $r-1$ on $\SSS^{(n-1,1)}$.  If $\pi$ is odd, then $\sgn (\pi) = -1$ and thus the eigenvalues of $\pi$ on $\mathcal{S}^{(n-1,1)} \otimes \sgn$ are exactly the negatives of the eigenvalues of $\pi$ on $\SSS^{(n-1,1)}$, by (\ref{tensor_evalue}). 
\end{proof}

We finish this section with a two simple but key observations that will streamline our proofs below. 

\begin{prop} \label{same_even}
Let $\lambda \vdash n$ and let $\pi$ be an even permutation.  Then $\pi$ has eigenvalue 1 on $\SSS^\lambda$ if and only if $\pi$ has eigenvalue 1 on $\SSS^{\lambda'}$.
\end{prop}

\begin{proof}
Since $\pi$ is even, $\sgn(\pi) = 1$, and since $\SSS^{\lambda'} = \SSS^\lambda \otimes \sgn$, it follows from (\ref{tensor_evalue}) that the characteristic polynomials of $\pi$ on $\SSS^{\lambda'}$ and $\SSS^{\lambda}$ are identical.
\end{proof}

\begin{lem} \label{switch}
Let $\lambda \vdash n$ and let $\pi \in S_n$ be an odd permutation.  Then $\pi$ has 1 as an eigenvalue on $\SSS^{\lambda}$ if and only if $\pi$ has $-1$ as an eigenvalue on $\SSS^{\lambda'}$.
\end{lem}

\begin{proof}
Since $\pi$ is odd, $\sgn (\pi) = -1$.  By (\ref{tensor_evalue}), if $\lbrace \alpha_i \rbrace$ are the eigenvalues of $\pi$ on $\SSS^{\lambda}$, then $\lbrace -\alpha_i \rbrace$ are the eigenvalues of $\pi$ on $\SSS^{\lambda'}$, and the set $\lbrace -\alpha_i \rbrace$ contains 1 if and only if $\lbrace \alpha_i \rbrace$ contains $-1$.
\end{proof}

Equipped with this background information, we turn to the main business of the paper over the next sections.

\section{Gamma-Shaped Specht Modules}

In this section we study the $\SSS^{\lambda}$ for $\lambda = (n-k,1^k)$ and prove that they are all unisingular unless $k=1$ (in which case it is never unisingular) or $k=n-2$, in which case it is unisingular if and only if $n$ is even.  We start by proving these two statements.  Recall that we reserve the notation $m_\mu(x)$ for the characteristic polynomial of a permutation $\pi \in \mathcal{C}_\mu \subseteq S_n$ acting on $\mathcal{M}^{(n-1,1)}$.

\begin{prop} \label{(n-1,1)}
For $n \geq 2$, the Specht module $\SSS^{(n-1,1)}$ is not unisingular.
\end{prop}

\begin{proof}
Let $\mu = (n)$ and let $c_{(n)}(x)$ denote the characteristic polynomial of $n$-cycles on $\SSS^{(n-1,1)}$.  By (\ref{char_poly_std}), we have $m_{(n)}(x) = x^n-1$ and by (\ref{standard_decomp}), we have $m_{(n)}(x) = c_{(n)}(x)(x-1)$, whence $c_{(n)}(1) \ne 0$.
\end{proof}

\begin{thm}  \label{(2,1^{n-2})}
The Specht module $\SSS^{(2,1^{n-2})}$ is unisingular if and only if $n$ is even.
\end{thm}

\begin{proof}
By (\ref{sgn_decomp}) we have $\sgn \otimes \SSS^{(n-1,1)} = \SSS^{(2,1^{n-2})}$.  Thus, if $\pi \in S_n$ is an even permutation, the characteristic polynomials of $\pi$ on $\SSS^{(n-1,1)}$ and $\SSS^{(2,1^{n-2})}$ are identical.  

Let $n$ be an odd positive integer so that the $n$-cycles are even permutations.  By Proposition \ref{(n-1,1)}, the $n$-cycles on $\SSS^{(n-1,1)}$ (and therefore also on $\SSS^{(2,1^{n-2})}$, by Proposition \ref{same_even}) do not have 1 as an eigenvalue.  

Now suppose $n$ is an even positive integer. We break the remainder of the proof into two cases, treating the $n$-cycles separately.

\medskip

\noindent \textbf{\fbox{Case 1: $\pi$ is an $n$-cycle}} Recall that the characteristic polynomial of $\pi$ on $\mathcal{M}^{(n-1,1)}$ is $m_{(n)}(x) = x^n-1$.    By (\ref{tensor_evalue}), the eigenvalues of $\pi$ on $\mathcal{M}^{(n-1,1)} \otimes \sgn$ are the $-\zeta$, where $\zeta^n = 1$.  But since $n$ is even, if $m_{(n)}(\zeta) = 0$, then $m_{(n)}(-\zeta) = 0$ as well. Therefore, the characteristic polynomial of $\pi$ on $\mathcal{M}^{(n-1,1)} \otimes \sgn$ is also $x^n-1$, and admits the well-known factorization
\[
x^n -1 = (x+1)(x-1)\left(\sum_{j=0}^{n-1}x^j\right) \left(\sum_{j=0}^{n-1}(-x)^j \right).
\]
Since $\sgn$ acts nontrivially on $\pi$, and 
\begin{align} \label{long_decomp}
\mathcal{M}^{(n-1,1)} \otimes \sgn = \left(\SSS^{(n-1,1)} \oplus \mathbf{1} \right) \otimes \sgn  = \SSS^{(2,1^{n-2})} \oplus \sgn,
\end{align}
we have that the characteristic polynomial of $\pi$ on $\SSS^{(2,1^{n-2})}$ is 
\[
(x-1)\left(\sum_{j=0}^{n-1}x^j\right) \left(\sum_{j=0}^{n-1}(-x)^j \right),
\]
which vanishes at $x=1$.  

\medskip

\noindent \textbf{\fbox{Case 2: $\pi$ is not an $n$-cycle}} If $\pi$ is not an $n$-cycle, then there exists a partition $\mu = (\mu_1,\dots,\mu_r) \vdash n$ with $r>1$ such that $\pi \in \mathcal{C}_\mu$.  Write
\[
\pi = \sigma_1 \cdots \sigma_r
\]
where $\sigma_i$ is a $\mu_i$-cycle.  Write $m_\mu(x)$ for the characteristic polynomial of $\pi$ on $\mathcal{M}^{(n-1,1)}$ and $c_\mu(x)$ for the characteristic polynomial of $\pi$ on $\SSS^{(n-1,1)}$. By (\ref{char_poly_std}), $m_\mu(x)$ vanishes to order $r$ at $x=1$ and thus $c_\mu(x)$ vanishes to order $r-1$ at $x=1$, by (\ref{standard_decomp}).

If $\pi$ is an even permutation, then the characteristic polynomials of $\pi$ on $\mathcal{M}^{(n-1,1)} \otimes \sgn$ and $\SSS^{(2,1^{n-2})}$ agree with $m_\mu(x)$ and $c_\mu(x)$, respectively, by 
(\ref{standard_decomp}).  Since $r \geq 2$, it follows that $c_\mu(1)=0$, and hence $\pi$ has eigenvalue 1 on $\SSS^{(2,1^{n-2})}$.

If $\pi$ is an odd permutation, then we employ the identity (\ref{long_decomp}) again.  If we can show that $m_\mu(-1)=0$ (so that $\pi$ has eigenvalue $-1$ on $\mathcal{M}^{(n-1,1)}$), then $\pi$ will have eigenvalue 1 on $\mathcal{M}^{(n-1,1)} \otimes \sgn$, and therefore $\pi$ will have eigenvalue 1 on $\SSS^{(2,1^{n-2})}$.  But since $\pi$ is odd, at least one of the $\sigma_i$ is a cycle of even length $n_i$, and therefore $m_\mu(x)$ vanishes at $-1$. This completes the proof.
\end{proof}

We now consider the Specht modules $\SSS^{(n-k,1^k)}$ for $k=2,\dots,n-3$, the cases $k=1$ and $k=n-2$ having been completed in the two propositions above. We break this task into two parts:
\begin{itemize}
\item Case A: $2 \leq k \leq \lfloor \frac{n-1}{2} \rfloor$, and 
\item Case B: $\lceil \frac{n+1}{2} \rceil \leq k \leq n-3$, and
\end{itemize}
By doing this we are taking advantage of the fact that every Specht module in Case B is the conjugate of a Specht module in Case A. 

\begin{thm} \label{mainthm1}
Let $n \geq 7$ and fix $k \in \left[2, \lfloor\frac{n-1}{2} \rfloor \right]$.  Then $\SSS^{(n-k,1^k)}$ is unisingular.
\end{thm}

\begin{proof}
Let $\mu = (\mu_1,\dots,\mu_r) \vdash n$ and suppose $\pi \in \mathcal{C}_\mu \subseteq S_n$.  Write $\pi = \sigma_1 \cdots \sigma_r$, where the $\sigma_i$ are disjoint cycles of length $\mu_i$.  The characteristic polynomial  of $\pi$ on $\SSS^{(n-k,1^k)}$ is given by (\ref{ext_evalue}), with the individual $\alpha_i \in \Lambda_\mu$ from (\ref{Emu}).  Therefore, we must show that for every partition $\mu$ of $n$, we can choose $k$ elements of $\Lambda_\mu$ that multiply to 1.  The remainder of the proof is dedicated to this.

\medskip

\noindent \textbf{\fbox{Case 1: $r >k$}} Then $\pi$ has eigenvalue 1 of multiplicity $r-1\geq k$ on $\SSS^{(n-1,1)}$.  Therefore $\pi$ has eigenvalue 1 on $\SSS^{(n-k,1^k)}$ as well.  

\medskip 

\noindent \textbf{\fbox{Case 2: $r=1$}} Then the eigenvalues of $\pi$ on $\SSS^{(n-1,1)}$ are the nontrivial roots of $x^{n} - 1$.  Let $z \in \C$ be a primitive $n$-th root of unity.  If $k = 2\ell$ is even, then 
\[
\lbrace z,z^{n-1},z^2,z^{n-2},\dots,z^\ell, z^{n-\ell} \rbrace
\]
is a cardinality-$k$ subset of the roots of $x^n-1$, the product of whose elements is 1.  If $k=2\ell+1$ is odd, then 
\[
\lbrace z,z^{n-1},z^2,z^{n-2},\dots,z^{\ell-1}, z^{n-\ell +1}, z^\ell, z^{\ell + 1}, z^{n-2\ell -1} \rbrace
\]
is a cardinality-$k$ subset of the roots of $x^n-1$, the product of whose elements is 1.  (Note that since $\ell = \lfloor k/2 \rfloor$ and $k \leq \lfloor (n-1)/2 \rfloor$, the elements of the sets above are all distinct.)

\medskip

\noindent \textbf{\fbox{Case 3: $r=k$}} The average value of each $\mu_i$ is $n/r$ and, since the $\mu_i$ form a non-increasing sequence, it follows that $\mu_1 \geq n/r$.  In the specific case $r=k$, and since $k< n/2$, we must have  $\mu_1 \geq 3$.  Therefore, if $z$ is a primitive $\mu_1$-th root of unity, then $z$ and $z^{-1}$ are nonequal and each occurs as an eigenvalue of $\pi$ on $\SSS^{(n-1,1)}$.  It is also the case that $1$ is an eigenvalue of $\pi$ of multiplicity $r-1$ = $k-1$.  Therefore, $\pi$ has eigenvalue 1 on $\SSS^{(n-k,1^k)}$ since we can select $k$ eigenvalues on $\SSS^{(n-1,1)}$ as follows:
\[
z\cdot z^{-1} \cdot \underbrace{1 \cdots 1}_{k-2\text{ times}} = 1.
\]

\medskip

\noindent \textbf{\fbox{Case 4: $r \in [2,k-1]$}} Reasoning as in the previous case, we note that the average value of each $\mu_i$ is $n/r$, whence $\mu_1 + \cdots + \mu_p \geq pn/r$.  Each $\sigma_i$ contributes $\mu_i - 1$ nontrivial  $\mu_i$th roots of unity as eigenvalues on $\SSS^{(n-1,1)}$.  If $\mu_i$ is odd, then these roots of unity multiply to 1 and if $\mu_i$ is even, then they multiply to $-1$.  In the case where $\mu_i$ is even, the product of all but $-1$ multiplies to 1.  Recall also that 1 is an eigenvalue of $\pi$ of multiplicity $r-1$.  Now our task is to deduce that there is a cardinality-$k$ set of eigenvalues on $\SSS^{(n-1,1)}$ that multiplies to 1.   We now have two subcases to consider.

\medskip

\noindent \textbf{\fbox{Subcase 4a: $r$ is odd}}  In this case, take $p = (r+1)/2$.  Then $\sigma_1,\sigma_2,\dots,\sigma_p$ contribute at least 
\[
\sum_{i = 1}^{(r+1)/2} \mu_i - 2 \geq \frac{(r+1)}{2} \cdot \frac{n}{r} - 2 \cdot \frac{(r+1)}{2} = \frac{n}{2} + \frac{n}{2r} + r-1 > \frac{n}{2} + 1 - r- 1 = \frac{n}{2} -r
\]
eigenvalues, the product of which is 1 (the latter inequality follows since $r<k<n/2$).  Together with the eigenvalue 1 of multiplicity $r-1$, this is now a cardinality $n/2-1 \geq k$ set of eigenvalues, the product of which is 1.

\medskip

\noindent \textbf{\fbox{Subcase 4b: $r$ is even}} Now take $p = (r+2)/2$.  Proceeding as above, $\sigma_1,\dots,\sigma_p$ contribute at least 
\[
\sum_{i = 1}^{(r+2)/2} \mu_i - 2 \geq \frac{(r+2)}{2} \cdot \frac{n}{r} - 2 \cdot \frac{(r+2)}{2} >   \frac{n}{2} + 2 - r- 2 = \frac{n}{2} -r
\]
eigenvalues, the product of which is 1.  Again, together with the eigenvalue 1 of multiplicity $r-1$, this is now a cardinality $n/2-1 \geq k$ set of eigenvalues, the product of which is 1.
\end{proof}

Theorem \ref{mainthm1} shows that the Gamma-shaped Specht modules $\SSS^{(n-k,1^k)}$ for $k=2,\dots,\lfloor \frac{n-1}{2} \rfloor$ (``wide and short'') are unisingular.  It remains to check their conjugates (``narrow and tall'') for unisingularity.  Most of the difficult work of the proof of Theorem \ref{mainthm2} has already been done in the proof of Theorem \ref{mainthm1}. 

\begin{thm} \label{mainthm2}
Let $n \geq 7$ and fix $k \in \left[\lfloor\frac{n+1}{2}\rfloor,n-3 \right]$.  Then $\SSS^{(n-k,1^k)}$ is unisingular.
\end{thm}

\begin{proof}
For ease of notation, we set $\lambda = (n-k,1^k)$ and $\lambda' = (k+1,1^{n-k-1})$ for the remainder of this proof; we have $\SSS^\lambda = \SSS^{\lambda'} \otimes \sgn$. By Theorem \ref{mainthm1}, $\SSS^{\lambda'}$ is unisingular and by Proposition \ref{same_even} it suffices to consider odd permutations $\pi$. Finally, by Lemma \ref{switch} it suffices to prove $-1$ is an eigenvalue of $\pi$ on $\SSS^{\lambda'}$.

Suppose $\pi$ is an odd permutation and write $\pi = \sigma_1 \cdots \sigma_r$, corresponding to the partition $\mu = (\mu_1,\dots,\mu_r) \vdash n$.  Since $\pi$ is odd, an odd number of the $\mu_i$ are even.   
In the proof of Theorem \ref{mainthm1} we demonstrated that on $\SSS^{(n-1,1)}$, for odd $\pi$, there are always $k$ eigenvalues, at least one of which is 1, that multiply to $1$; denote these $k$ eigenvalues by $z_1,\dots,z_k$, and without loss of generality we may assume $z_1=1$. Since at least one of the $\mu_i$ is even, $-1$ is an eigenvalue of $\pi$ on $\SSS^{(n-1,1)}$.  Replace $z_1$ with $-1$.  Then $-\prod_{j=2}^k z_j = -1$ is an eigenvalue of $\pi$ on $\SSS^{\lambda'}$ and hence 1 is an eigenvalue of $\pi$ on $\SSS^\lambda$.
\end{proof}

Altogether, Proposition \ref{(n-1,1)} and Theorems \ref{(2,1^{n-2})}, \ref{mainthm1}, and \ref{mainthm2} constitute a proof of Theorem \ref{mainthm}.

\section{The module $\SSS^{(2^2,1^{n-4})}$} \label{penultimate} 

In this section we prove Theorem \ref{(2^2,1^{n-4})}.  Our approach is to use as much general linear algebra as possible to reduce the problem to calculations that are specific to symmetric group theory.  Once we have done this, we will give a very brief description of the permutation module theory that we require to complete our work.  

\subsection{Generalities}  Recall from \cite{cullinan} that $\SSS^{(n-2,2)}$ is always unisingular once $n \geq 4$.  This allows us to quickly prove that all even permutations have 1 as an eigenvalue on $\SSS^{(2^2,1^{n-4})}$.

\begin{lem} \label{basic_lem}
Let $\pi \in S_n$ be an even permutation.  Then $\pi$ has 1 as an eigenvalue on $\SSS^{(2^2,1^{n-4})}$. 
\end{lem}

\begin{proof}
Since $\SSS^{(2^2,1^{n-4})} = \SSS^{(n-2,2)'}$, and $\SSS^{(n-2,2)}$ is unisingular, $\pi$ has eigenvalue 1 on $\SSS^{(2^2,1^{n-4})}$ by Proposition \ref{same_even}.
\end{proof}

We have now reduced the problem of unisingularity of $\SSS^{(2^2,1^{n-4})}$ to the  problem of determining that odd permutations have $-1$ as an eigenvalue on $\SSS^{(n-2,2)}$.  

In general, the Specht modules have a natural basis of \textsf{polytabloids} \cite[Def.~2.3.2]{sagan}, while the permutation modules have a natural basis of \textsf{tabloids} \cite[Def.~2.1.4]{sagan}.  From a computational point of view, the tabloids are a much simpler object to work with.  We can take advantage of the decomposition (\ref{n-2specht}) in order to work explicitly with the tabloids.

Let us fix some notation for the final part of this subsection.  Fix a conjugacy class $\mathcal{C}_\mu$ of odd permutations of $S_n$ and let $\pi \in \mathcal{C}_\mu$.   Write $\mu = (\mu_1, \dots, \mu_r)$.   Define three characteristic polynomials as follows:
\begin{align*}
M_\mu(x) &:\text{ characteristic polynomial of $\pi$ on $\mathcal{M}^{(n-2,2)}$} \\
c_\mu(x) &:\text{ characteristic polynomial of $\pi$ on $\mathcal{S}^{(n-2,2)}$} \\
m_\mu(x) &:\text{ characteristic polynomial of $\pi$ on $\mathcal{M}^{(n-1,1)}$}.
\end{align*}
Our goal is to determine conditions under which $c_\mu(-1) = 0$.

\begin{defn}
Let $\mu = (\mu_1,\dots,\mu_r) \vdash n$.  Define $\mathsf{E}(\mu)$ to be the number of even $\mu_i$ comprising $\mu$.    
\end{defn}

\begin{lem} \label{std_mult}
With all notation as above, the multiplicity of $-1$ as a root of $m_\mu(x)$ is $\mathsf{E}(\mu)$.
\end{lem}

\begin{proof}
Since $\mathcal{C}_\mu$ consists of odd permutations, it must be the case that an odd number of the $\mu_i$ are even integers.  By (\ref{char_poly_std}), we have 
\[
m_\mu(x) = \prod_{i=1}^r x^{\mu_i} -1,
\]
and each $x^{\mu_i} -1$ vanishes at $-1$ (to order 1) if and only if $\mu_i$ is even.  Hence the multiplicity of $-1$ as a root of $m_\mu(x)$ is the number of even $\mu_i$ comprising $\mu$, as claimed.
\end{proof}

\begin{prop} \label{reduction}
With all notation as above, $c_\mu(-1)=0$ if and only if the multiplicity of $-1$ as a root of $M_\mu(x)$ is $> \mathsf{E}(\mu)$.
\end{prop}

\begin{proof}
By (\ref{n-2specht}), we have $M_\mu(x) = c_\mu(x)m_\mu(x)$.  Now apply Lemma \ref{std_mult}.
\end{proof}

We now give a brief, self-contained description of the tabloids which are the natural basis of $\mathcal{M}^{(n-2,2)}$.

\subsection{Tabloids} We follow the treatment of \cite[Ch.~2]{sagan}.  Let $n$ be a positive integer and let $\lambda  = (\lambda_1,\dots,\lambda_r) \vdash n$.  Two tableaux $t$ and $s$ of shape $\lambda$ are said to be \textsf{row-equivalent} if the corresponding rows contain the same elements.  Row equivalence is an equivalence relation and the equivalence classes are called $\lambda$-\textsf{tabloids}.  If $t$ is a $\lambda$-tableau, then we write $\mathbf{\lbrace t \rbrace}$ for the equivalence class (tabloid) containing $t$. 

We write $\lambda! = \prod \lambda_i!$.  The number of distinct $\lambda$-tabloids is $n!/\lambda!$.  If $\mathbf{\lbrace t_1 \rbrace},\dots,\mathbf{\lbrace t_k \rbrace}$ is a complete set of $\lambda$-tabloids, then the permutation module $\mathcal{M}^\lambda$ is defined as the $S_n$-module spanned by the $\mathbf{\lbrace t_i \rbrace}$:
\[
\mathcal{M}^\lambda = \C[\mathbf{\lbrace t_1 \rbrace},\dots,\mathbf{\lbrace t_k \rbrace}].
\]
We are primarily interested in doing explicit basis calculations when $\lambda  = (n-2,2)$, so we focus on this case only.  A quick calculation reveals that
\begin{align} \label{dim}
\dim \mathcal{M}^{(n-2,2)} = \frac{n!}{(n-2,2)!} = \binom{n}{2}.
\end{align}
If $t$ is a $\lambda$-tableau, then it is of the form
\[
t = 
\begin{tabular}{ccccccc}
$a$ & $b$ & $c$ & $*$ & $*$ & $*$ & $d$ \\
$i$ & $j$
\end{tabular}
\]
where the numbers filling the tableau are distinct and lie between $1$ and $n$.  Since  tabloids  are uniquely characterized by their row entries, we can denote $\mathbf{\lbrace t \rbrace}$ succinctly as
\[
\mathbf{\lbrace t \rbrace} =
\begin{tabular}{cc}
\hline
$i$ & $j$ \\
\hline
\end{tabular}.
\]
Note that $\overline{\underline{i\ \ j}} = \overline{\underline{j\ \ i}}$ due to row-equivalence, so we adopt the convention of writing $\overline{\underline{i\ \ j}}$ if and only if $i<j$.  The symmetric group $S_n$ acts on the tabloids element-wise.  We now give an illustrative example before returning to the proof of Theorem \ref{(2^2,1^{n-4})}.

\begin{exm}
Let $n=7$ and $\lambda = (5,2)$.  Then $\dim \mathcal{M}^{(5,2)} = 21$ with basis
\begin{center}
$\lbrace$
\begin{tabular}{cc}
\hline
$1$ & $2$ \\
\hline
\end{tabular}\ , \
\begin{tabular}{cc}
\hline
$1$ & $3$ \\
\hline
\end{tabular}\ , \ \dots \ , \
\begin{tabular}{cc}
\hline
$6$ & $7$ \\
\hline
\end{tabular} $\rbrace$.
\end{center}
Let $\mu = (7)$ and let $\pi = (1234567) \in \mathcal{C}_{(7)}$.  Then $\pi \left(\overline{\underline{1\ \ 2}}\right) = \overline{\underline{1\ \ 7}}$.  Continuing, one computes the orbits under the cyclic group $\langle \pi \rangle$ generated by $\pi$:
\begin{align*}
{\rm Orbit}_{\langle \pi \rangle} \left(\overline{\underline{1\ \ 2}} \right) &= \lbrace 
\overline{\underline{1\ \ 2}}\ ,\ 
\overline{\underline{1\ \ 7}}\ ,\ 
\overline{\underline{6\ \ 7}}\ ,\ 
\overline{\underline{5\ \ 6}}\ ,\ 
\overline{\underline{4\ \ 5}}\ ,\ 
\overline{\underline{3\ \ 4}}\ ,\ 
\overline{\underline{2\ \ 3}}\rbrace \\
{\rm Orbit}_{\langle \pi \rangle} \left(\overline{\underline{1\ \ 3}} \right) &= \lbrace 
\overline{\underline{1\ \ 3}}\ ,\ 
\overline{\underline{2\ \ 7}}\ ,\ 
\overline{\underline{1\ \ 6}}\ ,\ 
\overline{\underline{5\ \ 7}}\ ,\ 
\overline{\underline{4\ \ 6}}\ ,\ 
\overline{\underline{3\ \ 5}}\ ,\ 
\overline{\underline{2\ \ 4}}\rbrace \\
{\rm Orbit}_{\langle \pi \rangle} \left(\overline{\underline{1\ \ 4}} \right) &= \lbrace 
\overline{\underline{1\ \ 4}}\ ,\ 
\overline{\underline{3\ \ 7}}\ ,\ 
\overline{\underline{2\ \ 6}}\ ,\ 
\overline{\underline{1\ \ 5}}\ ,\ 
\overline{\underline{4\ \ 7}}\ ,\ 
\overline{\underline{3\ \ 6}}\ ,\ 
\overline{\underline{2\ \ 5}}\rbrace.
\end{align*}
Therefore, the characteristic polynomial of $\pi$ on $\mathcal{M}^{(5,2)}$ is 
\[
M_{(7)}(x) = (x^7-1)^{\binom{7}{2}/7} = (x^7-1)^3.
\]  
We will use similar reasoning to compute characteristic polynomials in the next section.
\end{exm}

\subsection{Proof of Theorem \ref{(2^2,1^{n-4})}} Now we return to the proof of Theorem \ref{(2^2,1^{n-4})}, which we divide into several steps.  For the remainder of this section, we exclusively consider $\mu \vdash n$ such that $\mathcal{C}_\mu$ consists of odd permutations $\pi$.   The following lemma will imply certain $\SSS^{(2^2,1^{n-4})}$ are not unisingular.

\begin{lem} \label{big_mult}
Let $n$ be an odd positive integer and let $\pi \in \mathcal{C}_{(n-2,2)}$.  Then the characteristic polynomial of $\pi$ on $\mathcal{M}^{(n-2,2)}$ is given by
\[
M_{(n-2,2)}(x) = (x^{n-2}-1)^{(n-3)/2}(x^{2n-4} -1)(x-1).
\]
\end{lem}

\begin{proof}
We prove this by direct computation.  For notational convenience, we will suppose $\pi = (12\cdots(n-2))((n-1)n)$; this choice will not affect the characteristic polynomial.

There are exactly $\binom{n-2}{2}$ tabloids $\overline{\underline{i \ \ j}}$ with $1 \leq i < j \leq n-2$.  The action of $\langle \pi \rangle$ on this set decomposes into orbits of length $n-2$.  Since $n-2$ is odd, there are exactly $\binom{n-2}{2}/(n-2) = (n-3)/2$ orbits.  Therefore,
\[
(x^{n-2}-1)^{(n-3)/2} \mid M_{(n-2,2)}(x).
\]
We also observe that the tabloid $\overline{\underline{(n-1)\ \ n}}$ is fixed by $\pi$.  Thus $(x-1) \mid M_{(n-2,2)}(x)$ as well.  

It remains to evaluate the action of $\pi$ on the remaining $2n-4$ tabloids 
\[
\left\{ \overline{\underline{i \ \ j}}  ~\mid~ 1 \leq i \leq n-2, \ \ n-1 \leq j \leq n \right\}.
\]
These tabloids form a single orbit under the action of $\pi$, hence $(x^{2n-4} -1) \mid M_{(n-2,2)}(x)$.  Since the degree of $M_{(n-2,2)}(x) = \dim \mathcal{M}^{(n-2,2)} = \binom{n}{2} = (n-2) \cdot (n-3)/2 +1 +  (2n-4) $, the proof is complete.
\end{proof}

\begin{cor} \label{nonunicor}
If $n$ is odd then $\SSS^{(2^2,1^{n-4})}$ is not unisingular.  
\end{cor}

\begin{proof}
Let $\pi \in \mathcal{C}_{(n-2,2)}$. By (\ref{char_poly_std}), $m_{(n-2,2)}(x) = (x^{n-2}-1)(x^2-1)$, 
which has $-1$ as a root with multiplicity $\mathsf{E}((n-2,2)) = 1$.  By Lemma \ref{big_mult}, $M_{(n-2,2)}(x)$ has $-1$ as a root with multiplicity 1 as well.  By Proposition \ref{reduction}, $c_{(n-2,2)}(-1) \ne 0$ . Thus, by Lemma \ref{switch}, $\pi$ does not have 1 as an eigenvalue on $\SSS^{(2^2,1^{n-4})}$, which completes the proof.
\end{proof}

\begin{prop} \label{char_div}
Let $n$ be a positive integer and $\mu = (\mu_1,\dots,\mu_r) \vdash n$.  Let $\pi \in \mathcal{C}_\mu$ and let $M_\mu(x)$ be the characteristic polynomial of $\pi$ on $\mathcal{M}^{(n-2,2)}$.  Fix a part $\mu_k$ of $\mu$. 
\begin{enumerate}
\item If $\mu_k \geq 2$ is even, then $\left(x^{\mu_k}-1\right)^{\mu_k/2 - 1}\left(x^{\mu_k/2} -1\right) \mid M_\mu(x)$, and
\item if $\mu_k \geq 3$ is odd, then $\left(x^{\mu_k} - 1\right)^{(\mu_k-1)/2} \mid M_\mu(x)$.
\end{enumerate}
\end{prop}

\begin{proof}
Write $\pi = \sigma_1 \cdots \sigma_r$ as a product of disjoint cycles according to the partition $\mu$. Fix a cycle $\sigma_k$ of length $\mu_k$ and suppose, for notational convenience, that $\sigma_k = (12\cdots\mu_k)$.  It is clear that 
\[
V_{\sigma_k} \ddef \C \left[ \left\{ \overline{\underline{i\ \ j}}\right\}_{1 \leq i < j \leq \mu_k}   \right]
\]
is a $\binom{\mu_k}{2}$-dimensional subspace of $\mathcal{M}^{(n-2,2)}$ that is stable under the cyclic group generated by $\pi$.  Now we count orbits, and treat the cases of even-vs-odd $\mu_k$ separately.

\medskip 

\noindent \textbf{\fbox{Case 1: $\mu_k \geq 2$ is even}} One checks that the following are $\mu_k/2-1$ distinct orbits, each of length $\mu_k$:
\begin{align*}
{\rm Orbit}_{\langle \pi \rangle} \left(\overline{\underline{1\ \ 2}} \right) &= \lbrace 
\overline{\underline{1\ \ 2}}\ ,\ 
\overline{\underline{1\ \ \mu_k}}\ , \dots
, \overline{\underline{2\ \ 3}}\rbrace \\
{\rm Orbit}_{\langle \pi \rangle} \left(\overline{\underline{1\ \ 3}} \right) &= \lbrace 
\overline{\underline{1\ \ 3}}\ ,\ 
\overline{\underline{2\ \ \mu_k}}\ ,\dots,  
\overline{\underline{2\ \ 4}}\rbrace \\
{\rm Orbit}_{\langle \pi \rangle} \left(\overline{\underline{1\ \ 4}} \right) &= \lbrace 
\overline{\underline{1\ \ 4}}\ ,\ 
\overline{\underline{3\ \ \mu_k}}\ ,\dots, 
\overline{\underline{2\ \ 5}}\rbrace \\
\vdots &= \vdots \\
{\rm Orbit}_{\langle \pi \rangle} \left(\overline{\underline{1\ \ \mu_k/2}} \right) &= \lbrace 
\overline{\underline{1\ \ \mu_k/2}}\ ,\ 
\overline{\underline{\mu_k/2-1\ \ \mu_k}}\ ,\dots, 
\overline{\underline{2\ \ \mu_k/2 +1}}\rbrace.
\end{align*}
The remaining tabloids form an orbit of length $\mu_k/2$:
\[
{\rm Orbit}_{\langle \pi \rangle} \left(\overline{\underline{1\ \ (\mu_k/2 + 1)}} \right) = \lbrace 
\overline{\underline{1\ \ (\mu_k/2 + 1)}}\ ,\ 
\overline{\underline{\mu_k/2\ \ \mu_k}}\ ,\dots, 
\overline{\underline{2\ \ \mu_k/2 +2}}\rbrace.
\]
Altogether, this accounts for $\mu_k/2 - 1$ orbits of length $\mu_k$ and one orbit of length $\mu_k/2$, whence
\[
p_{\sigma_k}(x) \ddef \left(x^{\mu_k}-1\right)^{\mu_k/2 -1} \left(x^{\mu_k/2}-1\right)
\]
divides the characteristic polynomial of $\sigma_k$ on $V_{\sigma_k}$.  But since $\dim V_{\sigma_k} = \binom{\mu_k}{2}$ and $p_{\sigma_k}(x)$ has degree $\binom{\mu_k}{2}$, it must be the full characteristic polynomial of $\sigma_k$ on $V_{\sigma_k}$, and thus divides $M_\mu(x)$.  

\medskip

\noindent \textbf{\fbox{Case 2: $\mu_k \geq 3$ is odd}} One checks that the following are $(\mu_k-1)/2$ distinct orbits, each of length $\mu_k$:
\begin{align*}
{\rm Orbit}_{\langle \pi \rangle} \left(\overline{\underline{1\ \ 2}} \right) &= \lbrace 
\overline{\underline{1\ \ 2}}\ ,\ 
\overline{\underline{1\ \ \mu_k}}\ , \dots
, \overline{\underline{2\ \ 3}}\rbrace \\
{\rm Orbit}_{\langle \pi \rangle} \left(\overline{\underline{1\ \ 3}} \right) &= \lbrace 
\overline{\underline{1\ \ 3}}\ ,\ 
\overline{\underline{2\ \ \mu_k}}\ ,\dots,  
\overline{\underline{2\ \ 4}}\rbrace \\
{\rm Orbit}_{\langle \pi \rangle} \left(\overline{\underline{1\ \ 4}} \right) &= \lbrace 
\overline{\underline{1\ \ 4}}\ ,\ 
\overline{\underline{3\ \ \mu_k}}\ ,\dots, 
\overline{\underline{2\ \ 5}}\rbrace \\
\vdots &= \vdots \\
{\rm Orbit}_{\langle \pi \rangle} \left(\overline{\underline{1\ \ (\mu_k + 1)/2}} \right) &= \lbrace 
\overline{\underline{1\ \ (\mu_k+1)/2}}\ ,\ 
\overline{\underline{(\mu_k-1)/2\ \ \mu_k}}\ ,\dots, 
\overline{\underline{2\ \ (\mu_k + 3)/2}}\rbrace.
\end{align*}
The vector-space-sum of these orbits is $\binom{\mu_k}{2}$-dimensional, hence is all of $V_{\sigma_k}$.  It follows that the characteristic polynomial of $\sigma_k$ acting on $V_{\sigma_k}$ is $\left(x^{\mu_k} -1\right)^{(\mu_k-1)/2}$, 
hence $\left(x^{\mu_k} -1\right)^{(\mu_k-1)/2}$ divides  $M_\mu(x)$, as claimed.
\end{proof}

Before proving the final, crucial proposition needed for Theorem \ref{(2^2,1^{n-4})}, we pause for an auxiliary lemma that we will put to immediate use.

\begin{lem} \label{aux_lem}
Let $m$ be an odd positive integer. Let $\pi \in S_{2m}$ be the product of $m$ disjoint transpositions.  Then the characteristic polynomial of $\pi$ on $\mathcal{M}^{(n-2,2)}$ is $(x-1)^{m}\left(x^2-1 \right)^{m^2-m}$, as claimed.
\end{lem}

\begin{proof}
Conjugating $\pi$ if necessary, it suffices to take $\pi = (12)(34)\cdots((2m-1)(2m))$.  The tabloids 
\[
\overline{\underline{1 \ \ 2}}\ ,\ \overline{\underline{3 \ \ 4}}\ ,\  \dots \ , \ \overline{\underline{2m-1 \ \ 2m}}
\]
are fixed by $\pi$; the remaining $\binom{2m}{2} - m$ are not.  Since $\pi$ has order 2, the remaining tabloids can be paired into 2-dimensional stable subspaces of $\mathcal{M}^{(n-2,2)}$ and the characteristic polynomial on each subspace is $x^2-1$.  There are $\left(\binom{2m}{2} - m\right)/2 = m^2 - m$ such subspaces.  Together with the $m$ one-dimensional fixed subspaces, we conclude that the characteristic polynomial of $\pi$ is
\[
(x-1)^{m}\left(x^2-1 \right)^{m^2-m},
\]
as claimed.
\end{proof}

\begin{prop} \label{uni_prop}
Let $\pi$ be an odd permutation and suppose $\pi \not \in \mathcal{C}_{(n-2,2)}$.  Then $\pi$ has $-1$ as an eigenvalue on $\SSS^{(n-2,2)}$.
\end{prop}

\begin{proof}
Let $\mu = (\mu_1, \dots, \mu_r) \vdash n$ and let $\pi = \sigma_1 \dots \sigma_r$ be an odd permutation with $\sigma_i$ of order $\mu_i$.  Let 
$M_\mu(x)$ be the characteristic polynomial of $\pi$ on $\mathcal{M}^{(n-2,2)}$  and recall that $\mathsf{E}(\mu)$ is the total number of even $\mu_i$ that comprise $\mu$.  Since $\pi$ is odd, $\mathsf{E}(\mu)$ is an odd positive number. By Proposition \ref{reduction}, if we can show that the multiplicity of $-1$ as a root of $M_\mu(x)$ is $> \mathsf{E}(\mu)$, we will be done.  By Proposition \ref{char_div}(1), for each even $\mu_i$ comprising $\mu$, $M_\mu(x)$ is divisible by $\left(x^{\mu_i} - 1\right)^{\mu_i/2 - 1}$.  Therefore, for each even $\mu_i$, $-1$ is a root of $M_\mu(x)$ of multiplicity $\geq \mu_i/2-1$. 

\medskip

\noindent \textbf{\fbox{Case 1: all even $\mu_i$ are $\geq 4$}} By Proposition \ref{char_div}, each $\mu_{i}$ contributes an eigenvalue $-1$ with multiplicity $\mu_{i}/2 -1 + 1 = \mu_{i}/2$.  Since each such $\mu_{i} \geq 4$, this is a total contribution of $\geq 2\mathsf{E}(\mu) > \mathsf{E}(\mu)$ eigenvalues equal to $-1$.  

\medskip

\noindent \textbf{\fbox{Case 2: at least one even $\mu_i = 2$}} We introduce additional notation for this portion of the proof.  Let $E(\mu) = N$ and partition the set of even $\mu_i$ into two subsets:
\[
\lbrace \mu_{i_1},\dots \mu_{i_k} \rbrace \qquad \text{and} \qquad \lbrace \mu_{i_{k+1}},\dots \mu_{i_N} \rbrace,
\]
where each entry in the former is $\geq 4$ and each in the latter equals 2.  

Using similar reasoning as in \textbf{Case 1} ,by Proposition \ref{char_div}, each $\mu_{i_j}$ from the former set contributes an eigenvalue $-1$ with multiplicity $\mu_{i_j}/2 -1 + 1 = \mu_{i_j}/2$.  Since each such $\mu_{i_j} \geq 4$, this is a total contribution of $\geq 2k$ eigenvalues that equal $-1$.  

For the latter set, each \emph{pair} of transpositions contributes eigenvalue $-1$ with multiplicity 2.  Indeed, the element $(ab)(cd) \in S_n$ acts on the six-dimensional space
\[
\C[\overline{\underline{a \ \ b}}\ ,\ \overline{\underline{c \ \ d}}\ ,\  \overline{\underline{a \ \ c}}\ , \ \overline{\underline{a \ \ d}}\ ,\ \overline{\underline{b \ \ c}}\ ,\  \overline{\underline{b \ \ d}}]
\]
with characteristic polynomial $(x-1)^2(x^2-1)^2$. Now we must show that in every situation that can arise, the total multiplicity of $-1$ as an eigenvalue is $ > \mathsf{E}(\mu) = N$.  Recall that $N$ is odd.

\medskip

\noindent \textbf{\fbox{Subcase 2a: $k$ is odd}}  If $k$ is odd, then $N-k$ is even.  In particular, the $\sigma_i$ for which $\mu_i=2$ contribute $-1$ as an eigenvalue with multiplicity at least 
\[
\frac{N-k}{2} \cdot 2  = N-k,
\]
by our previous remark.  The remaining even $\mu_i$ contribute $\geq 2k$ factors of $-1$ by previous arguments as well.  Taken together, we have eigenvalue $-1$ with multiplicity  $\geq N-k+2k = N+k >N$, since $k$ is odd.

\medskip

\noindent \textbf{\fbox{Subcase 2b: $k$ is even}} If $k$ is even, then $N-k$ is odd.  Now we apply Lemma \ref{aux_lem} to the product $\sigma_{i_{k+1}}\sigma_{i_{k+2}} \cdots \sigma_{i_{N}}$ and conclude that it contributes eigenvalue $-1$ with multiplicity
\[
(N-k)^2-(N-k) = N^2 + (-2k - 1)N + (k^2 + k).
\]
Together with $\geq 2k$ factors of $-1$ coming from the $\mu_i \geq 4$, we see that the total multiplicity of $-1$ as an eigenvalue is at least
\begin{align} \label{minus_mult}
N^2 + (-2k - 1)N + (k^2 + k) + 2k = N^2 + (-2k - 1)N + (k^2 + 3k).
\end{align}

If $k \geq 2$, then (\ref{minus_mult}) is always $>N$.  If $k=0$, then (\ref{minus_mult}) reduces to $N^2-N$.  Since $N$ is odd, we have $N\geq 1$.  When $N \geq 3$, we have $N^2 - N > N$.  Therefore, we have reduced the entire proof to the case where $k=0$ and $N=1$.

If $k=0$ and $N=1$, then $E(\mu) = 1$ and we must show that $\pi$ has eigenvalue $-1$ on $\mathcal{M}^{(n-2,2)}$ with multiplicity $\geq 2$.  Consider the factorization $\pi = \sigma_1 \cdots \sigma_r$.  Exactly one of the $\sigma_i$ is a transposition, and by the assumption that $k=0$ and $N=1$, it is the only odd permutation comprising $\pi$. Denote this transposition by $\sigma$ and note that $\mu_1$ is odd and $> 2$.  Further note that if $r=2$, then we are in the case $\pi \in \mathcal{C}_{(n-2,2)}$ which is excluded by hypothesis.  Therefore $r\geq 3$.  
We have two final cases to consider.

If $\mu_1 \geq \mu_2 > 2$, then by applying  Lemma \ref{big_mult} separately to the products $\sigma_1\sigma$ and $\sigma_2\sigma$, we are furnished with two divisors of the characteristic polynomial of $\mathcal{M}^{(n-2,2)}$:
\begin{align*}
&(x^{\mu_1} -1)^{(\mu_1-1)/2}(x^{2\mu_1} -1)(x-1),\text{ and} \\
&(x^{\mu_2} -1)^{(\mu_2-1)/2}(x^{2\mu_2} -1)(x-1).
\end{align*}
Together, these contribute $-1$ as a root with multiplicity $2 > 1 = N$. 

Finally, we have the case $\mu_1 \geq3$, $\mu_2=2$, and $\mu_i = 1$ for $i=3,\dots,r$.  Let us write $\sigma_1 = (1\cdots\mu_1)$, $\sigma = ((\mu_1+1)(\mu_1+2))$, and $\sigma_3 = (\mu_1+3)$.  Then $\sigma \sigma_3$ acts nontrivially on the tabloid $\overline{\underline{\mu_1 +1 \ \ \mu_1+3}}$, producing eigenvlaue $-1$ of multiplicity 1.  
 Identical reasoning as in the previous case shows that $\sigma_1\sigma$ also contributes eigenvalue $-1$ with multiplicity 1.  Together, this shows that $\pi$ has eigenvalue $-1$ with multiplicity $\geq 2$, and thus completes the proof.
\end{proof}

We finish this section by proving Theorem  \ref{(2^2,1^{n-4})}, restated here for convenience.

\begin{thm}[Theorem \ref{(2^2,1^{n-4})}]
Let $n \geq 5$.  Then the Specht module $\SSS^{(2^2,1^{n-4})}$ is unisingular if and only if $n$ is even.
\end{thm}

\begin{proof}
If $n$ is odd then $\SSS^{(2^2,1^{n-4})}$ is not unisingular by Corollary \ref{nonunicor}.  Let $n$ be even.  By Proposition \ref{uni_prop} every odd permutation of $S_n$ has $-1$ as an eigenvalue on $\SSS^{(n-2,2)}$, which implies every odd permutation of $S_n$ has 1 as an eigenvalue on $\SSS^{(2^2,1^{n-4})}$ by Lemma \ref{switch}.  If $\pi$ is an even permutation, then $\pi$ has 1 as an eigenvalue on $\SSS^{(2^2,1^{n-4})}$ by Lemma \ref{basic_lem}.  This completes the proof of the theorem.
\end{proof}

\section{Observations and Future Work}

Given a symmetric group $S_n$ and a Specht module $\SSS^\lambda$, it is possible to determine whether or not $\SSS^\lambda$ is unisingular purely through computer calculation.  The \textsf{GAP} \cite{gap} package \textsf{young.g}, available (with user instructions) at 
\begin{center}
\href{https://nickerson.org.uk/atlas/progs/young/index.html}{https://nickerson.org.uk/atlas/progs/young/index.html} 
\end{center}
allows one to compute explicit matrix generators for $S_n$ in the representation $\SSS^\lambda$.  Equipped with these explicit matrices, it is then routine to compute the conjugacy classes and their characteristic polynomials using any number of computer algebra packages, such as \textsf{GAP} or \textsf{Magma} \cite{magma}.

However, as $n$ increases, the number of Specht modules (as well as their dimensions) makes explicit computation unwieldy, if not practically impossible.  That said, we have determined exactly which $\SSS^\lambda$ are unisingular for all $S_n$ with $n\leq 10$ (with one exception).  We present our data in the table below with the following notes.

For each $n=3,\dots,10$, we only record the $\SSS^\lambda$ that are not unisingular.  If $\SSS^\lambda$ is not unisingular, then there must exist at least one partition $\mu \vdash n$ such that the elements of $\mathcal{C}_\mu$ do not have 1 as an eigenvalue.  We call such a $\mathcal{C}_\mu$ an \textsf{offending class}.  In the table below we list the $\lambda$ for which $\SSS^\lambda$ is not unisingular and, for each such $\SSS^\lambda$, all offending classes $\mathcal{C}_\mu$.  We also record the total number of irreducible representations of $S_n$; this is given by the well-known partition function $\mathcal{P}(n)$.  Some further notes:
\begin{itemize}
\item the modules $\SSS^{(1^n)}$ are never unisingular since they are characters that take the value $-1$ on odd permuations;
\item the modules $\SSS^{(n-1,1)}$ are never unisingular by Proposition \ref{(n-1,1)}, with offending class $\mathcal{C}_{(n)}$;
\item the modules $\SSS^{(2,1^{n-2})}$ are not unisingular if and only if $n$ is odd by Theorem \ref{(2,1^{n-2})}, with offending class $\mathcal{C}_{(n)}$;
\item the modules $\SSS^{(2^2,1^{n-4})}$ are not unisingular if and only if $n$ is odd by Theorem \ref{(2^2,1^{n-4})}, with offending class $\mathcal{C}_{(n-2,2)}$.
\end{itemize}
In order not to clutter the table, we omit these examples and only list the $\SSS^\lambda$ that do not fall into these families; we call such $\lambda$  \textsf{exceptional}. We only list the offending classes for the exceptional $\lambda$. 

\begin{rmk}
In reading the table we make one caveat.  The highest-dimensional Specht module for $S_{10}$ is $\SSS^{(4,3,2,1)}$, with 
\[
\dim \SSS^{(4,3,2,1)} = f^{(4,3,2,1)} = \frac{10!}{7\cdot 5^2\cdot3^3} = 768.
\]
It is beyond the capability of our hardware to compute the conjugacy classes (this involves working with matrices with 
$589824 = 768 \times 768$ elements).  This is the only Specht module that we were unable to compute explicitly.  We denote this in the table with an asterisk $(^*)$, since it is possible (though unlikely) that $\SSS^{(4,3,2,1)}$ is not unisingular. 
\end{rmk}

\begin{center}
\begin{tabular}{|r|r|r|r|r|}
\hline
$n$ & $\mathcal{P}(n)$ & Total Unisingular & Exceptional $\lambda$ & Offending $\mu$ \\
\hline
2 & 2 & 1 &&\\
\hline
3 & 3 & 1 & & \\
\hline
4& 5 &  3 && \\
\hline
5 & 7 & 3 && \\
\hline
6 & 11 & 8
& $(2,2,2)$ & $(6)$ \\
\hline
7 & 15& 11 &&\\
\hline
8 & 22 & 18 &
 $(4,4)$ & $(5,3)$ \\
&&& $(2,2,2,2)$ & $(5,3)$ \\
\hline
9 & 30 & 26 && \\
\hline
10 & 42&  39$^*$ & $(2,2,2,2,2)$ & $(5,3,2)$ \\
\hline
\end{tabular}
\end{center}

We do not have enough data to responsibly state it as a conjecture, so we will merely observe that in all the exceptional cases covered in the table, as well as for the non-exceptional $\SSS^{(n-1,1)}$ and  $\SSS^{(2^2,1^{n-4})}$, if a Specht module is not unisingular, then there is always exactly one offending class.  It would be interesting to pursue this as a separate, but adjacent question in a future work.

It would also be interesting to focus on other families of Specht modules where an inductive process may be helpful, similar to the case of the $\SSS^{(n-k,1^k)}$ above.  For example, consider \textsf{two-rowed partitions} $\lambda = (n-k,k)$. Using Young's Rule \cite[Thm.~2.11.2]{sagan}, we have the decomposition
\[
\mathcal{M}^{(n-k,k)} = \bigoplus_{j=0}^{k} \SSS^{(n-j,j)}.
\]
It may be feasible to work inductively with tabloids to study the multiplicity of 1 as an eigenvalue of $\mathcal{M}^{(n-k,k)}$ (similar to our work in Section \ref{penultimate}) in order to determine which $\SSS^{(n-k,k)}$ are unisingular.

\end{document}